\documentclass[a4paper,12pt]{article}
\usepackage{authblk}
\usepackage{amsfonts}
\usepackage{dsfont}
\usepackage{amssymb,amsthm,amsmath}
\usepackage[sort]{natbib}
\usepackage[toc,page]{appendix}
\usepackage{enumerate}
\usepackage{xcolor}
\usepackage[final]{pdfpages}
\usepackage{hyperref}
\hypersetup{
	breaklinks,
	colorlinks=true,
	linkcolor=blue,
	urlcolor=red,
	citecolor=blue}
\usepackage[shortlabels]{enumitem}
\usepackage[latin1]{inputenc}
\usepackage[english]{babel}
\usepackage[left=2.5cm,right=2.5cm,top=1.5cm,bottom=1.5cm]{geometry}
\usepackage{lipsum}
\usepackage{cleveref}

\newenvironment{class}[1][2020 Mathematics Subject Classification]{\textbf{#1.} }{}
\newenvironment{MC}[1][Key words and phrases]{\textbf{#1.} }{}

\newtheorem{theorem}{Theorem}[section]

\newtheorem{proposition}[theorem]{Proposition}

\theoremstyle{definition}

\theoremstyle{remark}
\newtheorem{remark}[theorem]{Remark}

\begin{document}
\title{\bfseries \Large On a class of Reflected Mean-Field Stochastic Differential Equations with jumps}
\def\correspondingauthor{\footnote{Corresponding author : mohammed.elhachemy@edu.uiz.ac.ma }}

\author{\small \itshape Mohammed Elhachemy \correspondingauthor{}}

\affil{\scriptsize \itshape  \color{blue} Faculty of Sciences Agadir, Ibn Zohr University, Morocco.}

{      \affil{\color{blue} \scriptsize \itshape   mohammed.elhachemy@edu.uiz.ac.ma}
  
   }
\maketitle
\begin{abstract}
 This paper investigates a class of Reflected Mean-Field Stochastic Differential Equations when the noise is driven by a Brownian motion and an independent Poisson measure. We prove the existence and uniqueness of solutions and provide moments estimates for the state processes. We apply our result to derive a Feynman-Kac formula for the solution of an Integral-Partial Differential Equation with Neumann boundary conditions.
\\\begin{MC} Mean-field, Reflected Stochastic Differential Equation, Jumps, Integral-PDE, Feynman-Kac formula, Neumann boundary condition.
\end{MC}\\
\begin{class} 35Q99,60H10, 60H05,60H30,60J60.\end{class}
\end{abstract}

\section{\large Introduction}\label{sec:introduction}

Stochastic Differential Equations (SDEs) with jumps received significant attention and continue to attract research interest, mainly due to their extensive applications in mathematical finance and control theory, see e.g. \citep{kunita1996stochastic,rong2006theory,oksendal2019stochastic} and references therein. The reflecting cases for these equations, examined by various authors, such as \citep{laukajtys2003penalization,kohatsu1992reflecting}, are used to model processes constrained within a domain with a boundary, wherein the solution is reflected in a specific direction upon contacting the boundary.	

Mean-field approaches in mathematics are crucial in multiple fields, including finance and game theory.  After the seminal research by Lasry and Lions \cite{lasry2007mean} in 2007, interest in Mean-Field SDEs (MF-SDEs) has increased in both forward and backward settings, see e.g., \citep{MFSDEJani,feng2021generalized,BUCKDAHN1,elhachemy2025mean}. MF-SDEs, also known as McKean-Vlasov equations, constitute a category of SDEs characterized by coefficients that depend on the distribution.  These equations have been rigorously examined and are evolving dynamically due to their extensive applications in Partial Differential Equations (PDE) \citep{li2018mean,BUCKDAHN2,BUCKDAHN3}, finance \cite{hafayed2015mean}, mean-field control \citep{ni2016mean,li2016controlled,li2016mean} and mean-field games theory \citep{wang2012mean,Ma2024}.

For an initial condition $(t, \zeta) \in [0, T] \times L^2 (\Omega, \mathcal{F}_t, \mathbb{P}; \overline{G})$, for all $ t\leq s \leq T$, consider the MF-SDE with jumps, 
\begin{multline}\label{eq0}
	X_{s}^{t,\zeta} = \zeta + \int_{t}^{s} \mathbb{E}'[b(r,(X_{r}^{t,\zeta})',X_{r}^{t,\zeta})] dr  + \int_{t}^{s} \mathbb{E}'[\sigma(r,(X_{r}^{t,\zeta})',X_{r}^{t,\zeta})] dW_r\vspace{0.1cm}\\
	+ \int_{t}^{s} \int_{E}^{}\mathbb{E}'[\gamma(r,(X_{r}^{t,\zeta})',X_{r}^{t,\zeta},e)]\tilde{N}(dr,de)
\end{multline}
where $W$ denotes a $d$-dimensional standard Brownian motion, $\tilde{N}$ represents a compensated Poisson random measure, and the interpretation of $\mathbb{E}'$ is
\begin{equation*}
		\mathbb{E}'\left[\mu(s,X'_s,X_s)\right](\omega) = \mathbb{E}'\left[\mu(s,X'_s,X_s(\omega))\right]
		= \int_{\Omega}\mu(\omega',\omega,s,X_s(\omega'),X_s(\omega)) \mathbb{P}(d\omega') \end{equation*}
Equation \eqref{eq0} appears in \cite{li2016controlled}, where it is only presented and not proved; the authors defer the justification of existence, uniqueness, and the corresponding estimates to \cite{BUCKDAHN2}. However, demonstrating the result with the presence of jumps in a mean-field framework is not straightforward. It is noteworthy that alternative versions of \eqref{eq0} have been addressed as in \citep{Hao20161,briand2020mean}.
		
In this paper, we are going to study a class of reflected MF-SDEs with jumps defined as follows
\begin{equation}\label{eq.01}
	\left\{
	\begin{array}{lll}
		\displaystyle (i)~	X_{s}^{t,\zeta} = \zeta + \int_{t}^{s} \mathbb{E}'[b(r,(X_{r}^{t,\zeta})',X_{r}^{t,\zeta})] dr  + \int_{t}^{s} \mathbb{E}'[\sigma(r,(X_{r}^{t,\zeta})',X_{r}^{t,\zeta})] dW_r\vspace{0.1cm}\\
		\displaystyle	\hspace{1.5cm}+ \int_{t}^{s} \int_{E}^{}\mathbb{E}'[\gamma(r,(X_{r}^{t,\zeta})',X_{r}^{t,\zeta},e)]\tilde{N}(dr,de)	+ \int_{t}^{s} D\varphi(X^{t,\zeta}_r) dA^{t,\zeta}_{s}, \ \ s\in [t,T]\vspace{0.2cm}\\
		\displaystyle (ii)~A^{t,\zeta}_{s} = \int_{t}^{s} \mathds{1}_{\left\lbrace X_{r}^{t,\zeta} \in \partial G\right\rbrace} dA^{t,\zeta}_r, \ A^{t,\zeta}_{s} \text{  is increasing.}
	\end{array}
	\right. 
\end{equation}
where $ G $ is an open, connected, bounded subset of $ \mathbb{R}^{d}$, the properties of which will be described hereafter. A solution to \eqref{eq.01} is pair of processes $(X_{s}^{t,\zeta},A_{s}^{t,\zeta})_{t\leq s\leq T}$. Here, the process $(A_{s}^{t,\zeta})_{t\leq s\leq T}$ is referred to as the the local time of $(X_{s}^{t,\zeta})_{t\leq s\leq T}$ on the boundary $\partial G$, and the condition $(ii)$  means that $A_{.}^{t,\zeta}$ increases only when $X_{.}^{t,\zeta}  \in \partial G$. Note that if $G=\mathbb{R}^d$ one can put $A_{}^{t,\zeta} \equiv 0$ and then \eqref{eq.01} becomes the MF-SDE with jumps \eqref{eq0}.
 
This study tackles a gap in the literature by examining particle states subject to mean-field interactions and domain constraints. The incorporation of jumps and reflections into this framework provides a mathematical structure that significantly transcends a mere extension of existing results. This framework enables us to address diverse issues, specifically those related to Integral-PDEs with Neumann boundary conditions, exemplified by the case in \cite{elhachemy2025mean}. 

Our primary contributions in this paper are threefold: we demonstrate the existence and uniqueness of the solution, provide moment estimates, and establish a link between this class of SDEs and the corresponding Integral-PDE with Neumann boundary conditions. Note that the inclusion of the integral with respect to the increasing process $A$, the jumps component, and the mean-field operator complicates the computations relative to classical SDEs. Utilizing a characteristic of the domain $G$ (see \eqref{boundary_conditions}) and an appropriate function (see \eqref{luap_f}) to address the Stieltjes integral within the Itô formula, we establish the existence and uniqueness of the solution through a fixed-point argument. We also give some bounds and error estimates for the solutions to \eqref{eq.01}, which provide a means to control both the solution itself and the distance between two solutions based on the difference in their associated initial conditions. These estimates will play a crucial role in further researches. Finally, to establish the link with a Neumann boundary value problem for an Integral-PDE, we derive a Markov property for the solution to this kind of SDEs.

To the best of our knowledge, this specific issue has not been addressed in the existing literature, making our conclusion particularly innovative.

More precisely, the rest of the paper is structured as follows. In Section~\ref{sec2} we introduce basic notations and assumptions. The existence and uniqueness of the solution are proved in Section~\ref{sec3}. In Section~\ref{sec4} we prove moment bounds and stability estimations. Finally, in Section~\ref{sec5} we provide a Feynman--Kac formula for the solution of an Integral-PDE with Neumann boundary condition.

\section{\large Preliminaries and assumptions}\label{sec2}
In this section, we present the mathematical notations and assumptions to be used in this paper. Let $T>0$ be a fixed time, and consider a probability space $(\Omega,\mathcal{F},\mathbb{P})$ carrying a standard $d$-dimensional Brownian
motion $(W_{t})_{t\leq T}$ and an independent martingale measure $(\tilde{N}_t)_{t\leq T}$ corresponding to a standard Poisson random
measure $N$ on $\mathbb{R}_+\times E$ where $E:=\mathbb{R}^k\setminus\{0\}$, $(k\geq1)$ is equipped with its Borel $\sigma$-algebra $\mathcal{E}$.
Namely, for any Borel measurable subset $\Lambda\in\mathcal{E}$ such that $\lambda(\Lambda)<\infty$, it holds $\tilde{N}_t(\Lambda):=N_t(\Lambda)-t\lambda(\Lambda)$ where $\lambda$ is assumed to be a $\sigma$-finite measure on $(E,\mathcal{E})$, satisfying $\int_{E}(1\wedge|e|^2)\lambda(de)<\infty$. We suppose that there is a sub-$\sigma$-field $\mathcal{F}_0 \subset \mathcal{F}$ such that
\begin{itemize}
	\item [(i)] the Brownian motion $W$ and the Poisson random measure $N$ are independent of $\mathcal{F}_0$, and
	\item [(ii)]  $\mathcal{F}_0$ is said to be "rich enough", i.e. $\mathcal{P}_2(\mathbb{R}^n) = \left\{\mathbb{P}_\nu, \nu \in \mathbb{L}^2(\mathcal{F}_0;\mathbb{R}^n) \right\}, n\geq 1$. Here, $\mathcal{P}_2(\mathbb{R}^n)$ denotes the set of probability measures on $(\mathbb{R}^n, \mathcal{B}(\mathbb{R}^n))$ with finite second moment and $\mathcal{B}(\mathbb{R}^n)$ is the Borel $\sigma$-field over $\mathbb{R}^n$, and
	\item [(iii)]$\mathcal{F}_0$ includes all $\mathbb{P}$-null subsets of $\mathcal{F}$.
\end{itemize}
By $\mathbb{F}=\{\mathcal{F}_t\}_{t\leq T}$ we denote the filtration generated by the Brownian motion $W$ and the Poisson
random measure $N$, augmented by  $\mathcal{F}_0$, i.e.,
\begin{itemize}
	\item [] $\mathcal{F}^0_t = \sigma\left\{W_s, N([0,s]\times\Lambda) \ | \ s\leq t, \Lambda\in \mathcal{E} \right\}$,
	\item [] $\mathcal{F}_t := \mathcal{F}^0_{t^+} \vee \mathcal{F}_0 \left(=\left(\bigcap_{s:s>t}\mathcal{F}^0_s\right)\vee\mathcal{F}_0 \right),  t\in [0,T]$.
\end{itemize}
We will denote by $| . |$ the Euclidean norm on $\mathbb{R}^{d}$. Let $\mathbb{R}^{d\times d}$ be the Hilbert space of all $d\times d$ matrices, with the inner product  $IJ^\top:=Tr[IJ^\top],$ for all $I,J \in \mathbb{R}^{d\times d}$, where $\top$ denotes the transpose of matrices. For a given right continuous with left limits (RCLL) process $(X_t)_{t\leq T}$, \ $X_{t-}=\lim\limits_{s\nearrow t}X_s$.

We now introduce the framework of the Mean-Field. Let $(\bar{\Omega},\bar{\mathcal{F}},\bar{\mathbb{P}})=(\Omega \times \Omega,\mathcal{F}\otimes \mathcal{F},\mathbb{P}\otimes \mathbb{P})$ be the (non-completed) product of $(\Omega,\mathcal{F},\mathbb{P})$ with  itself. Let us endow this product space with the filtration $\bar{\mathbb{F}} =\left\{\bar{\mathcal{F}}_t = \mathcal{F}\otimes \mathcal{F}_t, \ 0\leq t\leq T \right\}$. A random variable $\xi$ originally defined on $\Omega$ is extended canonically to $\bar{\Omega}$ by $\xi'(\omega',\omega)=\xi(\omega'), (\omega',\omega) \in \bar{\Omega} = \Omega\times \Omega$. For any $\hbar \in  \mathbb{L}^1(\bar{\Omega},\bar{\mathcal{F}},\bar{\mathbb{P}};\mathbb{R})$, the random variable $\hbar(.,\omega) : \Omega \to \mathbb{R}$ belongs to $\mathbb{L}^1(\Omega,\mathcal{F},\mathbb{P};\mathbb{R})$, $\mathbb{P}(d\omega)$-a.s.; we denote its expectation by $\mathbb{E}'\left[\hbar(.,\omega)\right]=\int_{\Omega} \hbar(\omega',\omega)\mathbb{P}(d\omega')$.

Let $ G $ be an open connected bounded subset of $ \mathbb{R}^{d} $, which is such that for a function $ \varphi \in C_{b}^{2}(\mathbb{R}^{d}), G = \{\varphi > 0\}, \partial G = \{\varphi = 0\}, $ and $ |D\varphi(x)| = 1, x \in \partial G $. Observe that in particular $\varphi, D\varphi$ and $D^2\varphi$ are bounded in $ \bar{G} $. Also, note that at any boundary point $ x \in \partial G, D\varphi(x) $ is a unit normal vector to the boundary, pointing towards the interior of $ G $.

Before proceeding, we make the following remark.
\begin{remark}\label{remark1}
\begin{enumerate}
	\item According to \citep{lions1984stochastic,saisho1987stochastic}, it is evident that the aforementioned assumptions regarding the domain $G$ entail that, for any $x \in \partial G$ and $x' \in \overline{G}$, there exists a constant $c_0>0$  such that
	\begin{equation} \label{boundary_conditions}
		\frac{1}{2c_0}	|x - x'|^{2} +  (x' - x) \cdot \left(D\varphi(x)\right) \geq 0. 
	\end{equation}
	\item If moreover $G$ is a convex domain of $\mathbb{R}^d$ then $c_0=+\infty$.
\end{enumerate}	
\end{remark}

Throughout this paper, the measurable functions $ \mu : \overline{\Omega} \times [0, T] \times \mathbb{R}^{d} \times \mathbb{R}^{d} \rightarrow \mathbb{R}^{d} $ and $ \sigma : \overline{\Omega} \times [0, T] \times \mathbb{R}^{d} \times \mathbb{R}^{d} \rightarrow \mathbb{R}^{d \times d} $ and $ \gamma : \overline{\Omega} \times [0, T] \times \mathbb{R}^{d} \times \mathbb{R}^{d} \times E \rightarrow \mathbb{R}^{d} $ satisfy the following assumptions:

\begin{itemize}
	\item[({\bfseries $\mathcal{H}.1$})] 
	For each fixed $(x,x',e) \in \mathbb{R}^{d} \times \mathbb{R}^{d} \times E$, $\mu(.,x',x), \sigma(.,x',x)$ and $\gamma(.,x',x,e)$ are continuous in $t$.
	\item[({\bfseries $\mathcal{H}.2$})] 
	There exists $C > 0$ such that for all $0 \le t \le T$, $x_1, x_1', x_2, x_2' \in \mathbb{R}^d$,
	\begin{equation*}
		|\mu(t, x_1', x_1) - \mu(t, x_2', x_2)| + |\sigma(t, x_1', x_1) - \sigma(t, x_2', x_2)|  \le C|x_1' - x_2'| + C|x_1 - x_2|.
	\end{equation*}
	\item[({\bfseries $\mathcal{H}.3$})] 	There exists $\rho : E \to \mathbb{R}^+$ with $\int_{E} \rho^2(e) \lambda(de)<+\infty$, such that, for all  $0 \le t \le T$, $x_1, x_1', x_2, x_2' \in \mathbb{R}^d$ and $e\in E$,
	\begin{equation*}
		 |\gamma(t, x_1', x_1, e) - \gamma(t, x_2', x_2, e)| \le \rho(e) \left(|x_1' - x_2'| + |x_1 - x_2|\right).
		\end{equation*}
	\item[({\bfseries $\mathcal{H}.4$})]  	There exists $C > 0$ such that for all $0 \le t \le T$, $x, x' \in \mathbb{R}^d$ and $e\in E$,
	\begin{itemize}
		\item [(i)] $\quad |\mu(t, x', x)| + |\sigma(t, x', x)| \le C(1 + |x| + |x'|)$,
		\item [(ii)] $\quad|\gamma(t, x', x, e)| \le C(1 + |x| + |x'|+|e|)$.
	\end{itemize}
\end{itemize}

\section{\large Existence and uniqueness of the solution}\label{sec3}
In this Section we prove the existence and uniqueness of the solution of \eqref{eq.01}. Our proof is based on the Banach fixed point theorem on the Fréchet space $\mathbb{H} $ of RCLL processes $ (X_{s})_{0\leq s \leq T} $ satisfying
$$
\mathbb{E}\left[ \sup_{0 \leq s \leq T} |X_{s}|^{4} \right] < \infty,$$
and	equipped with the semi-norms
$$
\|X\|_{s} := \mathbb{E} \left[ \sup_{0 \leq s \leq T} |X_{s}|^{4} \right]^{\frac{1}{4}}.$$ 

\begin{theorem}\label{existence_theorem}
	Suppose assumptions ({\bfseries $\mathcal{H}.1$})-({\bfseries $\mathcal{H}.2$})-({\bfseries $\mathcal{H}.3$})-({\bfseries $\mathcal{H}.4$}) and the condition \eqref{boundary_conditions} hold. Then, equation \eqref{eq.01} admits a unique solution.
\end{theorem}

\begin{proof} 
	Given $U \in \mathbb{H}$, it follows from \cite{kohatsu1992reflecting} that there exists a unique pair $(X^{t,\zeta}_s, A^{t,\zeta}_s)_{s\geq 0}$ satisfying
	\begin{equation}
		\left\{
		\begin{array}{lll}
			\displaystyle (i)~	X^{t,\zeta}_s =  \zeta + \int_{t}^{s} \mathbb{E}'[\mu(r,(U_r)',X^{t,\zeta}_r)] dr  + \int_{t}^{s} \mathbb{E}'[\sigma(r,(U_r)',X^{t,\zeta}_r)] dW_r\vspace{0.1cm}\\
			\displaystyle	\hspace{1.5cm}+ \int_{t}^{s}  \int_{E}^{}\mathbb{E}'[\gamma(r,(U_r)',X^{t,\zeta}_r,e)]\tilde{N}(dr,de)	+ \int_{t}^{s} D \varphi(X^{t,\zeta}_r) dA^{t,\zeta}_r,\vspace{0.2cm}\\
			\displaystyle (ii)~A^{t,\zeta}_s = \int_{t}^{s}\mathds{1}_{\left\lbrace X^{t,\zeta}_r \in \partial G\right\rbrace} dA^{t,\zeta}_r, \ A^{t,\zeta}_s \text{ is increasing.}
		\end{array}
		\right. \label{eq2}
	\end{equation}
	
	Consequently, we can define the mapping $ X^{t,\zeta} = F(U) : \mathbb{H} \rightarrow \mathbb{H} $. For any $U, V \in \mathbb{H}$, let $(X^{t,\zeta}_s, A^{t,\zeta}_s)_{s\geq 0}$ and $(Y^{t,\zeta}_s, K^{t,\zeta}_s)_{s\geq 0}$ represent the unique solutions of the aforementioned reflected SDE with jumps \eqref{eq2} associated with $U$ and $V$, respectively.
	
	We denote
	\[
	\begin{aligned}
		\overline{\mu}_r &= \mathbb{E}'[\mu(r,({U}_r)',X^{t,\zeta}_r)], \\
		\underline{\mu}_r &= \mathbb{E}'[\mu(r,(V_r)',Y^{t,\zeta}_r)], \\
		\overline{\sigma}_r &= \mathbb{E}'[\sigma(r,(U_r)',X^{t,\zeta}_r)], \\
		\underline{\sigma}_r &= \mathbb{E}'[\sigma(r,(V_r)',Y^{t,\zeta}_r)], \\
		\overline{\gamma}_r(e) &= \mathbb{E}'[\gamma(r,(U_r)',X^{t,\zeta}_r,e)], \\
		\underline{\gamma}_r(e)&= \mathbb{E}'[\gamma(r,(V_r)',Y^{t,\zeta}_r,e)].
	\end{aligned}
	\]
	Note that, in order to overcome the term with reflection, we are going to apply Itô's formula (see \cite{protter2005stochastic}, Theorem 33, pp. 81) to the following function
\begin{equation}\label{luap_f}
		\exp \left\{ - \frac{1}{c_0} (\varphi (X^{t,\zeta}_s) + \varphi (Y^{t,\zeta}_s)) \right\} \times |X^{t,\zeta}_{s} - Y^{t,\zeta}_{s}|^2.
\end{equation}
Due to the complexity of this function, we devised it. First, we have
	\begin{equation*}
		\begin{aligned}
			d|X^{t,\zeta}_s - Y^{t,\zeta}_s|^2 &= \left[  2(X^{t,\zeta}_{s} - Y^{t,\zeta}_{s}) \cdot (\overline{\mu}_s - \underline{\mu}_s) +
			(\overline{\sigma}_s - \underline{\sigma}_s)(\overline{\sigma}_s - \underline{\sigma}_s)^\top    + \int_E |\overline{\gamma}_s(e) - \underline{\gamma}_s(e)|^2 \, \lambda(de) \right] ds\\
			&\quad + 2(X^{t,\zeta}_{s} - Y^{t,\zeta}_{s}) \cdot (\overline{\sigma}_s - \underline{\sigma}_s) dW_s  \\
			& \quad+ 2(X^{t,\zeta}_{s} - Y^{t,\zeta}_{s}) \cdot \left( D\varphi(X^{t,\zeta}_s) dA^{t,\zeta}_s - D\varphi(Y^{t,\zeta}_s) dK_s \right) \\		 
			&\quad + \int_E \left[ |X^{t,\zeta}_{s-} - Y^{t,\zeta}_{s-} + \overline{\gamma}_s(e) - \underline{\gamma}_s(e)|^2 - |X^{t,\zeta}_{s-} - Y^{t,\zeta}_{s-}|^2 \right] \tilde{N}(ds, de)
		\end{aligned}
	\end{equation*}	
	and since $ \varphi \in C_{b}^{2}(\mathbb{R}^{d})$, it follows that
	\begin{equation} \label{dp_formula}
		\begin{aligned}
			d\varphi(X^{t,\zeta}_s) &= \left[D \varphi(X^{t,\zeta}_{s}) \cdot \overline{\mu}_s  + \frac12  \overline{\sigma}_s \overline{\sigma}_s^\top D^2 \varphi(X^{t,\zeta}_{s}) \big) \right. \\
			&\left.\quad+ \int_E \Big[ \varphi(X^{t,\zeta}_{s-} + \overline{\gamma}_s(e)) - \varphi(X^{t,\zeta}_{s-}) - D \varphi(X^{t,\zeta}_{s-}) \cdot \overline{\gamma}_s(e) \Big] \lambda(de) \right]ds\\
			&\quad	+ D \varphi(X^{t,\zeta}_{s}) \cdot \overline{\sigma}_s dW_s  + dA^{t,\zeta}_s  + \int_E \big[ \varphi(X^{t,\zeta}_{s-} + \overline{\gamma}_s(e)) - \varphi(X^{t,\zeta}_{s-}) \big] \tilde{N}(ds, de) 
		\end{aligned}
	\end{equation}	
	and similarly we get the same for $\varphi(Y^{t,\zeta}_s)$ with $\underline{\mu}_s, \underline{\sigma}_s, \underline{\gamma}_s(e), dK^{t,\zeta}_s$. 
	
	Then, by \eqref{dp_formula}, we get
		\[
		\begin{aligned}
				d\big(\varphi(X^{t,\zeta}_t) + \varphi(Y^{t,\zeta}_t)\big) &= \Bigg[ D\varphi(X^{t,\zeta}_{t}) \cdot \overline{\mu}_t + D\varphi(Y^{t,\zeta}_{t}) \cdot \underline{\mu}_t + \frac12  \overline{\sigma}_t \overline{\sigma}_t^\top D^2\varphi(X^{t,\zeta}_{t})  + \frac12  \underline{\sigma}_t \underline{\sigma}_t^\top D^2\varphi(Y^{t,\zeta}_{t})  \\
				&\hspace{1cm}+ \int_E \Big[ \varphi(X^{t,\zeta}_{t-} + \overline{\gamma}_t(e)) - \varphi(X^{t,\zeta}_{t-}) - D\varphi(X^{t,\zeta}_{t-}) \cdot \overline{\gamma}_t(e) \\
				&\hspace{1cm} + \varphi(Y^{t,\zeta}_{t-} + \underline{\gamma}_t(e)) - \varphi(Y^{t,\zeta}_{t-}) - D\varphi(Y^{t,\zeta}_{t-}) \cdot \underline{\gamma}_t(e) \Big] \lambda(de) \Bigg] dt \\
				&\hspace{-1cm} + \Big( D\varphi(X^{t,\zeta}_{t}) \cdot \overline{\sigma}_t + D\varphi(Y^{t,\zeta}_{t}) \cdot \underline{\sigma}_t \Big) dW_t + (dA^{t,\zeta}_t + dK^{t,\zeta}_t) \\
				&\hspace{-1cm} + \int_E \left[ \left(\varphi(X^{t,\zeta}_{t-} + \overline{\gamma}_t(e)) - \varphi(X^{t,\zeta}_{t-}) \right)  +  \left(\varphi(Y^{t,\zeta}_{t-} + \underline{\gamma}_t(e)) - \varphi(Y^{t,\zeta}_{t-})\right)\right] \tilde{N}(dt, de).
			\end{aligned}
		\]
	Therefore, if we set $$H_s:=H(X^{t,\zeta}_s , Y^{t,\zeta}_s) =	\exp \left\{ - \frac{1}{c_0} (\varphi (X^{t,\zeta}_s) + \varphi (Y^{t,\zeta}_s)) \right\},$$
	we get 
	\[
	\begin{aligned}
		dH_s &= H_{s-} \times \Bigg\{ -\frac{1}{c_0} \Big( D\varphi(X^{t,\zeta}_{s}) \cdot \overline{\mu}_s + D\varphi(Y^{t,\zeta}_{s}) \cdot \underline{\mu}_s \Big) -\frac{1}{2c_0} \Big( \overline{\sigma}_s \overline{\sigma}_s^\top D^2\varphi(X^{t,\zeta}_{s})  + \underline{\sigma}_s \underline{\sigma}_s^\top D^2\varphi(Y^{t,\zeta}_{s}) \Big)  \\
		&\qquad + \frac{1}{2c_0^2} \Big( \big( D\varphi(X^{t,\zeta}_{s}) \cdot \overline{\sigma}_s \big)^2 + \big( D\varphi(Y^{t,\zeta}_{s}) \cdot \underline{\sigma}_s \big)^2 + 2 \big( D\varphi(X^{t,\zeta}_{s}) \cdot \overline{\sigma}_s \big) \big( D\varphi(Y^{t,\zeta}_{s}) \cdot \underline{\sigma}_s \big) \Big)  \\
		&\qquad + \int_E \Bigg[ \exp\Bigg(-\frac{1}{c_0}\Big( \varphi(X^{t,\zeta}_{s-} + \overline{\gamma}_s(e)) - \varphi(X^{t,\zeta}_{s-}) + \varphi(Y^{t,\zeta}_{s-} + \underline{\gamma}_s(e)) - \varphi(Y^{t,\zeta}_{s-}) \Big)\Bigg) - 1 \\
		&\qquad\qquad + \frac{1}{c_0} \Big( D\varphi(X^{t,\zeta}_{s-}) \overline{\gamma}_s(e)  + D\varphi(Y^{t,\zeta}_{s-})  \underline{\gamma}_s(e) \Big) \Bigg] \lambda(de) \Bigg\}ds \\
		&\qquad -\frac{1}{c_0} H_{s} \times \Big( D\varphi(X^{t,\zeta}_{s}) \cdot \overline{\sigma}_s + D\varphi(Y^{t,\zeta}_{s}) \cdot \underline{\sigma}_s \Big) dW_s \\
		&\qquad -\frac{1}{c_0} H_{s} \times \big( dA^{t,\zeta}_s + dK_s \big) \\
		&\qquad +H_{s-} \int_E \Bigg[ \exp\Bigg(-\frac{1}{c_0}\Big( \varphi(X^{t,\zeta}_{s-} + \overline{\gamma}_s(e)) - \varphi(X^{t,\zeta}_{s-}) + \varphi(Y^{t,\zeta}_{s-} + \underline{\gamma}_s(e)) - \varphi(Y^{t,\zeta}_{t-}) \Big)\Bigg) - 1 \Bigg] \tilde{N}(ds, de)
	\end{aligned}
	\]
Ultimately, by employing integration by parts (see to \cite{protter2005stochastic}, Corollary 2, pp. 68) on the function \eqref{luap_f}, we derive 
	\[
	\begin{aligned}
		H_{s} \times |X^{t,\zeta}_{s} - Y^{t,\zeta}_{s}|^2 = &\int_{t}^s H_{r}  \Biggl\{-\frac{1}{c_0} |X^{t,\zeta}_{r} - Y^{t,\zeta}_{r}|^2 \left[ D\varphi(X^{t,\zeta}_{r}) \cdot \overline{\mu}_r  + D\varphi(Y^{t,\zeta}_{r})\cdot \underline{\mu}_r\right.\\
		&\hspace{-1.5cm}\left.+ \frac{1}{2} \overline{\sigma}_r \overline{\sigma}_r^\top D^2\varphi(X^{t,\zeta}_{r}) + \frac{1}{2}\underline{\sigma}_r \underline{\sigma}_r^\top D^2\varphi(Y^{t,\zeta}_{r}) \right] -\frac{1}{2c_0}\left( D\varphi(X^{t,\zeta}_{r}) \cdot \overline{\sigma}_r  + D\varphi(Y^{t,\zeta}_{r})\cdot \underline{\sigma}_r\right)^2\\
		&\hspace{-1.5cm}-\frac{2}{c_0} (X^{t,\zeta}_{r} - Y^{t,\zeta}_{r})(\overline{\sigma}_r - \underline{\sigma}_r) \left\{ D\varphi(X^{t,\zeta}_{r}) \cdot \overline{\sigma}_r  + D\varphi(Y^{t,\zeta}_{r})\cdot \underline{\sigma}_r\right\}\\
		&\hspace{-1.5cm}+2(X^{t,\zeta}_{r} - Y^{t,\zeta}_{r})(\overline{\mu}_r - \underline{\mu}_r) + (\overline{\sigma}_r - \underline{\sigma}_r)(\overline{\sigma}_r - \underline{\sigma}_r)^\top \Biggr\}dr\\
		&\hspace{-1.5cm} +\int_{t}^s \int_{E} H_{r-} \Biggl\{ \frac{1}{c_0} |X^{t,\zeta}_{r-} - Y^{t,\zeta}_{r-}|^2 \left( D\varphi(X^{t,\zeta}_{r-}) \cdot \overline{\gamma}_r(e)  + D\varphi(Y^{t,\zeta}_{r-})\cdot \underline{\gamma}_r(e)\right)\\
		&\hspace{-1.5cm} + \left(\exp\left(-\frac{1}{c_0} \Psi(r)\right)-1 \right) \Big[ |X^{t,\zeta}_{r-} - Y^{t,\zeta}_{r-} + \overline{\gamma}_r(e) - \underline{\gamma}_r(e)|^2 \Big]+ |\overline{\gamma}_r(e) - \underline{\gamma}_r(e)|^2 \Biggr\}\lambda(de)dr\\
		&\hspace{-1.5cm} +\int_{t}^s H_{r} \left\{ 2(X^{t,\zeta}_{r} - Y^{t,\zeta}_{r})(\overline{\sigma}_r - \underline{\sigma}_r) -\frac{1}{c_0} |X^{t,\zeta}_{r} - Y^{t,\zeta}_{r}|^2 \left( D\varphi(X^{t,\zeta}_{r}) \cdot \overline{\sigma}_r  + D\varphi(Y^{t,\zeta}_{r})\cdot \underline{\sigma}_r\right) \right\}dW_r\\
		&\hspace{-1.5cm} + \int_{t}^s H_{r} \left\{2(X^{t,\zeta}_{r} - Y^{t,\zeta}_{r}) D\varphi(X^{t,\zeta}_r)  -\frac{1}{c_0} |X^{t,\zeta}_{r} - Y^{t,\zeta}_{r}|^2 \right\} dA^{t,\zeta}_r	 \\
		&\hspace{-1.5cm} - \int_{t}^s H_{r} \left\{2(X^{t,\zeta}_{r} - Y^{t,\zeta}_{r}) D\varphi(Y^{t,\zeta}_r)  +\frac{1}{c_0} |X^{t,\zeta}_{r} - Y^{t,\zeta}_{r}|^2 \right\} dK_r\\
		&\hspace{-1.5cm}	+\int_{t}^s \int_{E}  H_{r-} \left\{ \exp\left(-\frac{1}{c_0} \Psi(r)\right) |X^{t,\zeta}_{r-} - Y^{t,\zeta}_{r-} + \overline{\gamma}_r(e) - \underline{\gamma}_r(e)|^2 - |X^{t,\zeta}_{r-} - Y^{t,\zeta}_{r-}|^2 \right\}\tilde{N}(dr, de)
	\end{aligned}
	\]
	with $ \Psi(r) = \varphi(X^{t,\zeta}_{r-} + \overline{\gamma}_r(e)) - \varphi(X^{t,\zeta}_{r-}) + \varphi(Y^{t,\zeta}_{r-} + \underline{\gamma}_r(e)) - \varphi(Y^{t,\zeta}_{r-})$.
	\\Remark that, by \eqref{boundary_conditions}, we have 
	\begin{align*}
		&\int_{t}^s  H_{r} \left\{2(X^{t,\zeta}_{r} - Y^{t,\zeta}_{r}) D\varphi(X^{t,\zeta}_r)  -\frac{1}{c_0} |X^{t,\zeta}_{r} - Y^{t,\zeta}_{r}|^2 \right\} dA^{t,\zeta}_r \leq 0,\\
		&\int_{t}^s  H_{r} \left\{2(Y^{t,\zeta}_{r} - X^{t,\zeta}_{r}) D\varphi(Y^{t,\zeta}_r)  -\frac{1}{c_0} |X^{t,\zeta}_{r} - Y^{t,\zeta}_{r}|^2 \right\}dK^{t,\zeta}_r \leq 0. 
	\end{align*}
	By Doob's inequality and considering the boundedness of the exponential term, $\varphi$, as well as its first and second derivatives, $D\varphi$ and $D^2\varphi$, and the assumptions regarding the coefficients, there exists a constant $C > 0$ such that for any $t \in [0,T],$
	\begin{align*}
		\mathbb{E} \left[\sup_{t \leq r \leq s}|X^{t,\zeta}_{r} - Y^{t,\zeta}_{r}|^4 \right] &\leq C \mathbb{E}\left[\int_{t}^{s} \left(|X^{t,\zeta}_{r} - Y^{t,\zeta}_{r}|^4 + \mathbb{E}'|(U_{r})' - (V_{r})'|^4  \right) dr\right]\\
		&\leq C\mathbb{E}\left[\int_{t}^{s} \sup_{t \leq u \leq r} |X^{t,\zeta}_{u} - Y^{t,\zeta}_{u}|^4 du\right] + C\mathbb{E}\left[\int_{t}^{s} \sup_{t \leq u \leq r} |U_{u} - V_{u}|^4  du\right].
	\end{align*}
	By Gronwall's lemma, yields
	\begin{align*}
		\mathbb{E} \left[\sup_{t \leq r \leq s}|X^{t,\zeta}_{r} - Y^{t,\zeta}_{r}|^4 \right] &\leq C \int_{t}^{s} \mathbb{E}\left[ \sup_{t \leq u \leq r} |U_{u} - V_{u}|^4  \right]du,
	\end{align*}
	and thus
	\begin{align*}
		\mathbb{E} \left[\sup_{t \leq r \leq s}|F(U_r) - F(V_r)|^4 \right] &\leq C \int_{t}^{s} \mathbb{E}\left[ \sup_{t \leq u \leq r} |U_{u} - V_{u}|^4  \right]du.
	\end{align*}
	Finally, by applying a standard Picard iteration argument, we establish the existence of a unique fixed point for the operator $F$ in $\mathbb{H}$. Since the set $G$ is uniformly bounded, any solution to \eqref{eq.01} necessarily belongs to $\mathbb{H}$. Consequently, this ensures the well-posedness and pathwise uniqueness of the solution to \eqref{eq.01}. The proof is complete.
\end{proof}

\section{\large Moments Estimates}\label{sec4}
In this part of our paper, we state some moments estimates of the solution $(X^{t,\zeta}_s,A^{t,\zeta}_s)_{t\leq s\leq T}$ of \eqref{eq.01}. 
\begin{theorem}\label{prop estim X} For all $ t \in [0,T]$, and $\zeta,\zeta' \in\mathbb{L}^2(\Omega,\mathcal{F}_t,\mathbb{P};\bar{G})$, There exists a constant $C>0$ such that
	\begin{itemize}
		\item [(i)]\quad$\mathbb{E} \Big[ \underset{s\in [t,T]}{\sup} \left|X^{t,\zeta}_s -X^{t,\zeta'}_s\right|^4\Big|\mathcal{F}_t \Big]\leq C \left|\zeta-\zeta'\right|^4,$ 
		\item [(ii)]\quad$ \mathbb{E} \Big[\underset{s\in [t,T]}{\sup} \left|A^{t,\zeta}_s -A^{t',\zeta'}_s\right|^4\Big|\mathcal{F}_t \Big] \leq C \left|\zeta-\zeta'\right|^4.$
	\end{itemize}
	If moreover we assume that $G$ is convex, we have for a constant $C>0$
		\begin{itemize}
		\item [(iii)] \quad$\mathbb{E} \Big[ \underset{s\in [t,t+\theta]}{\sup} \left|X^{t,\zeta}_s-\zeta \right|^2 \Big| \mathcal{F}_t \Big]\leq C\theta^2, \quad a.s.$, for any $0\leq \theta \leq T-t$.
		\item [(iv)] \quad$\mathbb{E} \Big[\left|A^{t,\zeta}_{t+\theta}\right|^2 \Big| \mathcal{F}_t \Big]\leq C\theta^2, \quad a.s.$, for any $0\leq \theta \leq T-t$, 
		\item [(v)]\quad$ \mathbb{E} \Big[ e^{\mu A^{t,\zeta}_s}\Big|\mathcal{F}_t\Big]\leq C(\mu,s),\quad \forall s\in [t,T]$, $\quad \forall \mu >0$.	
	\end{itemize}
\end{theorem}
\begin{proof} (i). For any $ t\leq s \leq T$, applying Itô formula to  \[ \exp \left\{ - \frac{1}{c_0} (\varphi (X^{t,\zeta}_s) + \varphi (X^{t,\zeta'}_s)) \right\} \times |X^{t,\zeta}_s - X^{t,\zeta'}_s|^2, 	\]
	we derive as in the proof of theorem \ref{existence_theorem}
	\begin{equation*}
		\mathbb{E} \Big[ \underset{ t \leq  s\leq T }{\sup} \left|X^{t,\zeta}_s -X^{t,\zeta'}_s\right|^4\Big|\mathcal{F}_t \Big]\leq C\left|\zeta-\zeta'\right|^4 + C\mathbb{E}\left[\int_{t}^{T} \left|X^{t,\zeta}_s -X^{t,\zeta'}_s\right|^4  ds\right]
	\end{equation*}
	Consequently, by virtue of Gronwall's lemma, we have the desired outcome.  \\
	(ii). From \eqref{dp_formula}, we have 
	\begin{equation} \label{A_formula}
		\begin{aligned}
			A^{t,\zeta}_s =	&\varphi(X^{t,\zeta}_s) - \varphi(\zeta) \\
			&- \int_{t}^{s}\Big[D \varphi(X^{t,\zeta}_r) \cdot \mathbb{E}'[\mu(r,(X^{t,\zeta}_r)',X^{t,\zeta}_r)]  \\
			&\hspace{1.5cm}+ \frac12 \mathbb{E}'[\sigma(r,(X^{t,\zeta}_r)',X^{t,\zeta}_r)]  \mathbb{E}'[\sigma^\top(r,(X^{t,\zeta}_r)',X^{t,\zeta}_r)] D^2 \varphi(X^{t,\zeta}_r) \Big] dr\\
			&\left.- \int_{t}^{s} \int_E \Big[ \varphi(X^{t,\zeta}_{r-} + \mathbb{E}'[\gamma(r,(X^{t,\zeta}_{r-})',X^{t,\zeta}_{r-},e)])  - \varphi(X^{t,\zeta}_{r-}) \right. \\
			&\left.\hspace{1.5cm} - D \varphi(X^{t,\zeta}_{r-}) \cdot \mathbb{E}'[\gamma(r,(X^{t,\zeta}_{r-})',X^{t,\zeta}_{r-},e)] \Big] \lambda(de) dr\right.\\
			&	- \int_{t}^{s} D \varphi(X^{t,\zeta}_r) \cdot \sigma(r,(X^{t,\zeta}_r)',X^{t,\zeta}_r) dW_r   \\
			&+ \int_{t}^{s}\int_E \big[ \varphi(X^{t,\zeta}_{r-} + \mathbb{E}'[\gamma(r,(X^{t,\zeta}_{r-})',X^{t,\zeta}_{r-},e)]) - \varphi(X^{t,\zeta}_{r-}) \big] \tilde{N}(dr, de) 
		\end{aligned}
	\end{equation}	
	From Burkholder-Davis-Gundy's inequality (see \cite{protter2005stochastic}, Theorem 48, pp.195) and the properties of $\varphi, \mu, \sigma$ and $\gamma$, we have, for $t\leq s\leq T$,
	\begin{multline*}
		\mathbb{E} \Big[\underset{ t \leq  s\leq T }{\sup} \left|A^{t,\zeta}_s -A^{t,\zeta'}_s\right|^4\Big|\mathcal{F}_t \Big]\leq C\left|\zeta-\zeta'\right|^4 + C\mathbb{E}\left[\left|X^{t,\zeta}_s -X^{t,\zeta'}_s\right|^4\Big|\mathcal{F}_t\right] \\
		+ C\mathbb{E}\left[\int_{t}^{T} \left|X^{t,\zeta}_s -X^{t,\zeta'}_s\right|^4  ds\Big|\mathcal{F}_t\right]
	\end{multline*}
	From (i) and Gronwall's inequality we get the desired result.\\
	(iii). For any $\zeta' \in\mathbb{L}^2(\Omega,\mathcal{F}_t,\mathbb{P};\bar{G})$, by Itô Formula, we get
	\begin{equation*}
		\begin{aligned}
			|X^{t,\zeta}_s-\zeta'|^2 &=  |\zeta-\zeta'|^2 +2 \int_{t}^{s} \left(X^{t,\zeta}_r-\zeta'\right) \cdot \mathbb{E}'[\mu(r,(X^{t,\zeta}_r)',X^{t,\zeta}_r)] dr \\
			&+ \int_{t}^{s}
			\mathbb{E}'[\sigma(r,(X^{t,\zeta}_r)',X^{t,\zeta}_r)]\mathbb{E}'[\sigma^\top(r,(X^{t,\zeta}_r)',X^{t,\zeta}_r)] dr  
		\\&	+ \int_{t}^{s} \int_E |\mathbb{E}'[\gamma(r,(X^{t,\zeta}_{r-})',X^{t,\zeta}_{r-},e)]|^2 \, \lambda(de)  dr \\
			&+ 2 \int_{t}^{s} \left(X^{t,\zeta}_r-\zeta'\right)  \cdot \mathbb{E}'[\sigma(r,(X^{t,\zeta}_r)',X^{t,\zeta}_r)] dW_r  \\
			&+ 2 \int_{t}^{s} \left(X^{t,\zeta}_r-\zeta'\right)  \cdot  D\varphi(X^{t,\zeta}_r) dA^{t,\zeta}_r  	 
			\\
			&+ \int_{t}^{s} \int_E \left[ |X^{t,\zeta}_{r-}-\zeta'  + \mathbb{E}'[\gamma(r,(X^{t,\zeta}_{r-})',X^{t,\zeta}_{r-},e)]|^2 - |X^{t,\zeta}_{r-}-\zeta'|^2 \right] \tilde{N}(dr, de).
		\end{aligned}
	\end{equation*}
	Since $G$ is supposed convex (see Remark \ref{remark1}), we have \[ \int_{t}^{s} \left(X^{t,\zeta}_r-\zeta'\right)  \cdot  D\varphi(X^{t,\zeta}_r) dA^{t,\zeta}_r\leq 0. \]
	Then, for $\theta>0$, we get
	\[ \mathbb{E} \Big[\underset{ s\in [t,t+\theta]}{\sup} \left|X^{t,\zeta}_s -\zeta'\right|^2\Big|\mathcal{F}_t \Big]\leq C\left(\left|\zeta-\zeta'\right|^2 +\theta\right). \]
	Moreover, due to the properties of $G$, we obtain namely,
	\[ \mathbb{E} \Big[\underset{ s\in [t,t+\theta]}{\sup} \left|X^{t,\zeta}_s -\zeta\right|^2\Big|\mathcal{F}_t \Big]\leq C\theta. \]\\
	(iv). By \eqref{A_formula}, we have
	\begin{equation*}
		\begin{aligned}
			A^{t,\zeta}_s \leq	&\left|\varphi(X^{t,\zeta}_s) - \varphi(\zeta)\right| + C \int_{t}^{s} \left(1+\left|X^{t,\zeta}_r\right|^2+\mathbb{E}\left|X^{t,\zeta}_r\right|^2\right) dr\\
			&\quad + \left| \int_{t}^{s} D \varphi(X^{t,\zeta}_r) \cdot \sigma(r,(X^{t,\zeta}_r)',X^{t,\zeta}_r) dW_r \right|  \\
			&\quad + \left| \int_{t}^{s}\int_E \big[ \varphi(X^{t,\zeta}_{r-} + \mathbb{E}'[\gamma(r,(X^{t,\zeta}_{r-})',X^{t,\zeta}_{r-},e)]) - \varphi(X^{t,\zeta}_{r-}) \big] \tilde{N}(dr, de) \right|
		\end{aligned}
	\end{equation*}
	and furthermore, from Burkholder-Davis-Gundy inequality and (iii), we get
	\begin{equation*}
		\mathbb{E} \Big[\left|A^{t,\zeta}_s\right|^2\Big|\mathcal{F}_t \Big]\leq C 
		\mathbb{E} \Big[\underset{ s\in [t,t+\theta]}{\sup} \left|X^{t,\zeta}_s-\zeta\right|^2\Big|\mathcal{F}_t \Big] +C\theta \leq C\theta.
	\end{equation*}\\
	(v). Is a direct consequence of (iv).
\end{proof}

\section{\large Associated Integral-PDE with Neumann boundary condition}\label{sec5}
In this section, we establish the connection between our equation \eqref{eq.01} and a problem involving Integral-PDE with Neumann boundary conditions. Initially, note that due to the uniqueness of the solution of \eqref{eq.01}, we obtain the subsequent flow property. 
\begin{equation}\label{flow}
	X^{s,X^{t,\zeta}_s}_r =	X^{t,\zeta}_r, \qquad r \in [s,T],  \ \ \forall \ 0 \leq t\leq s \leq T, \ \zeta \in\mathbb{L}^2(\Omega,\mathcal{F}_t,\mathbb{P};\bar{G}).
\end{equation}
Consequently, via \eqref{flow} and the analogous reasoning employed in the proof of Theorem 7.1.2 in \cite{oksendal2013stochastic}, we establish the Markov property. 
\begin{proposition}
Let $g :\mathbb{R}^d \to \mathbb{R}$ be a bounded Borel measurable function. Then, for the solution $(X^{t,\zeta}_s,A^{t,\zeta}_s)_{t\leq s\leq T}$ of \eqref{eq.01}, it holds that
\begin{equation*}
	\mathbb{E}\left(g(X^{t,\zeta}_s)\Big| \mathcal{F}_t \right)= \mathbb{E}\left(g(X^{s,X^{t,\zeta}_s}_s) \right), \qquad 0\leq t\leq s \leq T.
\end{equation*}
\end{proposition}
Therefore, we can construct a semi-group corresponding to this Markov process.

Now, for any data $(t,\zeta)\in [0,T]\times  \mathbb{L}^2(\Omega,\mathcal{F}_t,\mathbb{P};\bar{G})$ and an arbitrary $x_0 \in \bar{G}$, we consider the following reflected MF-SDE with jumps,
	\begin{equation}\label{008}
		\left\{
		\begin{array}{lll}
			\displaystyle (i) ~ X^{t,\zeta} _s = \zeta + \int_{t}^{s} \mathbb{E}'[\mu(r,(X^{0,x_0}_r)',X^{t,\zeta}_r)] dr 	 + \int_{t}^{s}  \mathbb{E}'[\sigma(r,(X^{0,x_0}_r)',X^{t,\zeta}_r)] dW_r\vspace{0.1cm}\\
		\displaystyle	\hspace{1.1cm}+ \int_{t}^{s}  \int_{E}^{}\mathbb{E}'[\gamma(r,(X^{0,x_0}_r)',X^{t,\zeta}_r,e)]\tilde{N}(dr,de)	+ \int_{t}^{s} D \varphi(X^{t,\zeta}_r) dA^{t,\zeta}_r,\vspace{0.2cm}\\
		\displaystyle	(ii)~A^{t,\zeta}_s = \int_{t}^{s} \mathds{1}_{\left\lbrace X^{t,\zeta}_r \in \partial G\right\rbrace }dA^{t,\zeta}_r.
				\end{array}
			\right.
		\end{equation}
Observe that Theorem~\ref{existence_theorem} guarantees the existence and uniqueness of the process $X^{0,x_0}$ in \eqref{008}. With $X^{0,x_0}$ at hand, equation \eqref{008} reduces to a reflected SDE with jumps. Consequently, $(X^{t,\zeta}_s,A^{t,\zeta}_s)_{t\leq s\leq T}$ is well-defined, where $A^{t,\zeta}_\cdot$ is a continuous increasing process.

Now, without loss of generality, we put $\zeta = x$ where $x \in \bar{G}$. Subsequently, we obtain the following representation.
\begin{proposition}
 Let $(X^{t,x}_s,A^{t,x}_s)_{t\leq s\leq T}$ be the solution of \eqref{008} and $v \in	C_b^{1,2}\left([0,T]\times \bar{G}\right)$ be a solution to the following integral-PDE with Neumann boundary condition 
	\begin{equation}\label{ipde}
		\left\{
		\begin{array}{lll}
			\displaystyle	 \frac{\partial v}{\partial t}(t,x)+ \mathcal{L}v(t,x)=0, \hspace{1.2cm} &t\in [0,T[, x\in G \\\\
			\displaystyle v(T,x)=f(x), &x\in \bar{G}\\\\
			\displaystyle	\frac{\partial v}{\partial n}(t,x)=0, & t\in [0,T[, x\in \partial G.
		\end{array}
		\right.
	\end{equation}
	where $f$ is a bounded function and $\mathcal{L} $ is the associated integro-differential operator such that
	\begin{equation*}
		\begin{split}
			&\mathcal{L} = \mathcal{R} + \mathcal{S},\\
			&\mathcal{R}v(t,x)=\frac{1}{2}\mathbb{E}[\sigma(t,X_t^{0,x_0},x)]. \mathbb{E}[\sigma^\top(t,X_t^{0,x_0},x)].D^2v(t,x) + \mathbb{E}[\mu(t,X_t^{0,x_0},x)].D v(t,x),\\
			&\mathcal{S}v(t,x)=\int_{E}^{}\big[v(t,x+\mathbb{E}[\gamma(t,X_t^{0,x_0},x,e)])-v(t,x)- D v(t,x) .\mathbb{E}[\gamma(t,X_t^{0,x_0},x,e)]\big]\lambda(de).
		\end{split}
	\end{equation*}
	and $$ \frac{\partial v}{\partial n} = \sum_{i=1}^{d} \frac{\partial v}{\partial x_i}  \frac{\partial \varphi}{\partial x_i},\quad \forall x\in \partial G.$$
Then, we have the following representation
\begin{equation}\label{representation}
	v(t,x)=\mathbb{E}\left[f(X^{t,x}_T) \Big| \mathcal{F}_t \right].
\end{equation}
\end{proposition}		
\begin{proof}
By Itô's formula, we get
\begin{equation}\label{ito_formula_f}
	\begin{array}{lll}
\displaystyle f(X^{t,x}_T) =   v(T,X^{t,x}_T)&\displaystyle= v(t,x) + \int_{t}^{T} \left(\frac{\partial v}{\partial r}(r,X^{t,x}_r)+ \mathcal{L}v(r,X^{t,x}_r)\right)dr\\
& \displaystyle + \int_{t}^{T} Dv(t,X^{t,x}_r).D\varphi(X^{t,x}_r)dA^{t,\zeta}\\
& \displaystyle + \int_{t}^{T} \mathbb{E}'[\sigma(t,X_t^{0,x_0},X^{t,x}_r)].Dv(X^{t,x}_r)dW_r\\
& \displaystyle + \int_{t}^{T} \int_{E}  \big[ v\left(r,X^{t,x}_{r-} + \mathbb{E}'[\gamma(r,(X^{0,x_0}_r)',X^{t,x}_r,e)]\right) - v(r,X^{t,x}_{r-}) \big] \tilde{N}(dr, de)
	\end{array} 
\end{equation}
The second and the third terms in the right hand side of \eqref{ito_formula_f} are equal to zero by \eqref{ipde}. Next we take the conditional expectation on both sides of \eqref{ito_formula_f}. Given the conditions on $v$, $\sigma$ and $\gamma$ we obtain
\begin{equation*}
\mathbb{E}\left[\int_{t}^{T} \mathbb{E}'[\sigma(t,X_t^{0,x_0},X^{t,x}_r)].Dv(X^{t,x}_r)dW_r \Big| \mathcal{F}_t \right] =0,
\end{equation*}
and 
\begin{equation*}
	\mathbb{E}\left[\int_{t}^{T} \int_{E}  \big[ v\left(r,X^{t,x}_{r-} + \mathbb{E}'[\gamma(r,(X^{0,x_0}_r)',X^{t,x}_r,e)]\right) - v(r,X^{t,x}_{r-}) \big] \tilde{N}(dr, de) \Big| \mathcal{F}_t \right] =0.
\end{equation*}
Hence, we have the representation \eqref{representation}.
\end{proof}

\begin{remark}
For further research directions, other extensions would consider the reflected MF-SDE with jumps by considering Stieltjes integration with respect to increasing processes that are not necessarily continuous. Another extension is to combine our result in this paper with the Mean-Field Generalized BSDE with jumps discussed in \cite{elhachemy2025mean} to address an obstacle problem of nonlocal Integral-PDE with nonlinear Neumann boundary conditions.
\end{remark}

\bibliography{refCS}

@incollection{kunita1996stochastic,
	title={Stochastic differential equations with jumps and stochastic flows of diffeomorphisms},
	author={Kunita, Hiroshi},
	booktitle={It{\^o}'s stochastic calculus and probability theory},
	pages={197--211},
	year={1996},
	publisher={Springer}
}

@book{rong2006theory,
	title={Theory of stochastic differential equations with jumps and applications: mathematical and analytical techniques with applications to engineering},
	author={Rong, SITU},
	year={2006},
	publisher={Springer Science \& Business Media}
}

@incollection{oksendal2019stochastic,
	title={Stochastic Control of Jump Diffusions},
	author={{\O}ksendal, Bernt and Sulem, Agn{\`e}s},
	booktitle={Applied Stochastic Control of Jump Diffusions},
	pages={93--155},
	year={2019},
	publisher={Springer}
}

@article{laukajtys2003penalization,
	title={Penalization methods for reflecting stochastic differential equations with jumps},
	author={{\L}aukajtys, Weronika and S{\l}omi{\'n}ski, Leszek},
	journal={Stochastics and Stochastic Reports},
	volume={75},
	number={5},
	pages={275--293},
	year={2003},
	publisher={Taylor \& Francis}
}

@phdthesis{kohatsu1992reflecting,
	title={Reflecting stochastic differential equations with jumps},
	author={Kohatsu-Higa, Arturo},
	year={1992},
	school={Purdue University}
}

@article{lasry2007mean,
	title={Mean field games},
	author={Lasry, Jean-Michel and Lions, Pierre-Louis},
	journal={Japanese journal of mathematics},
	volume={2},
	number={1},
	pages={229-260},
	year={2007},
	publisher={Springer}
}

@article{MFSDEJani,
	title = {Mean-field stochastic differential equations with a discontinuous diffusion coefficient},
	author = {Jani, Nyk\"anen},
	journal = {Probability, Uncertainty and Quantitative Risk},
	volume = {8},
	number = {3},
	pages = {351-372},
	year = {2023}
}

@article{feng2021generalized,
	title={Generalized mean-field backward stochastic differential equations and related partial differential equations},
	author={Feng, Xinwei},
	journal={Applicable Analysis},
	volume={100},
	number={16},
	pages={3299--3321},
	year={2021},
	publisher={Taylor \& Francis}
}

@article{BUCKDAHN1,
	author = {Rainer Buckdahn and Boualem Djehiche and Juan Li and Shige Peng},
	title = {{Mean-field backward stochastic differential equations: A limit approach}},
	volume = {37},
	journal = {The Annals of Probability},
	number = {4},
	publisher = {Institute of Mathematical Statistics},
	pages = {1524-1565},
	year = {2009}
}

@article{BUCKDAHN2,
	title = {Mean-field backward stochastic differential equations and related partial differential equations},
	journal = {Stochastic Processes and their Applications},
	volume = {119},
	number = {10},
	pages = {3133-3154},
	year = {2009},
	author = {Rainer Buckdahn and Juan Li and Shige Peng}
}

@article{BUCKDAHN3,
	author = {Rainer Buckdahn and Juan Li and Shige Peng and Catherine Rainer},
	title = {{Mean-field stochastic differential equations and associated {PDE}s}},
	volume = {45},
	journal = {The Annals of Probability},
	number = {2},
	publisher = {Institute of Mathematical Statistics},
	pages = {824-878},
	year = {2017}
}

@article{li2018mean,
	title={Mean-field forward and backward SDEs with jumps and associated nonlocal quasi-linear integral-PDEs},
	author={Li, Juan},
	journal={Stochastic Processes and their Applications},
	volume={128},
	number={9},
	pages={3118--3180},
	year={2018},
	publisher={Elsevier}
}

@article{hafayed2015mean,
	title={Mean-field maximum principle for optimal control of forward--backward stochastic systems with jumps and its application to mean-variance portfolio problem},
	author={Hafayed, Mokhtar and Tabet, Moufida and Boukaf, Samira},
	journal={Communications in Mathematics and Statistics},
	volume={3},
	number={2},
	pages={163--186},
	year={2015},
	publisher={Springer}
}

@article{ni2016mean,
	title={Mean-field stochastic linear--quadratic optimal control with Markov jump parameters},
	author={Ni, Yuan-Hua and Li, Xun and Zhang, Ji-Feng},
	journal={Systems \& Control Letters},
	volume={93},
	pages={69--76},
	year={2016},
	publisher={Elsevier}
}

@article{li2016controlled,
	title={Controlled mean-field backward stochastic differential equations with jumps involving the value function},
	author={Li, Juan and Min, Hui},
	journal={Journal of Systems Science and Complexity},
	volume={29},
	number={5},
	pages={1238--1268},
	year={2016},
	publisher={Springer}
}

@article{li2016mean,
	title={Mean-field stochastic linear quadratic optimal control problems: closed-loop solvability},
	author={Li, Xun and Sun, Jingrui and Yong, Jiongmin},
	journal={Probability, Uncertainty and Quantitative Risk},
	volume={1},
	number={1},
	pages={2},
	year={2016},
	publisher={Springer}
}

@article{wang2012mean,
	title={Mean field games for large-population multiagent systems with Markov jump parameters},
	author={Wang, Bing-Chang and Zhang, Ji-Feng},
	journal={SIAM Journal on Control and Optimization},
	volume={50},
	number={4},
	pages={2308--2334},
	year={2012},
	publisher={SIAM}
}

@article{Ma2024,
	title = {$\epsilon$-Nash mean-field games for stochastic linear-quadratic systems with delay and applications},
	author = {Heping Ma and Yu Shi and Ruijing Li and Weifeng Wang},
	journal = {Probability, Uncertainty and Quantitative Risk},
	volume = {9},
	number = {3},
	pages = {389-404},
	year = {2024}
}

@ARTICLE{Hao20161,
	author = {Hao, Tao and Li, Juan},
	title = {Mean-field {SDE}s with jumps and nonlocal integral-{PDE}s},
	year = {2016},
	journal = {Nonlinear Differential Equations and Applications},
	volume = {23},
	number = {2},
	pages = {1-51}
}

@article{elhachemy2025mean,
	title={Mean-Field Reflected Generalized BSDEs with Jumps Under Stochastic Conditions},
	author={Elhachemy, Mohammed},
	journal={Mediterranean Journal of Mathematics},
	volume={22},
	number={3},
	pages={64},
	year={2025},
	publisher={Springer}
}

@article{lions1984stochastic,
	author    = {Lions, Pierre-Louis and Sznitman, Alain-Sol},
	title     = {Stochastic differential equations with reflecting boundary conditions},
	journal   = {Communications on Pure and Applied Mathematics},
	volume    = {37},
	number    = {4},
	pages     = {511--537},
	year      = {1984}
}

@book{protter2005stochastic,
	title =     {Stochastic integration and differential equations},
	author =    {Philip E. Protter},
	publisher = {Springer},
	isbn =      {3540003134,9783540003137},
	year =      {2004},
	series =    {Applications of mathematics 21 0172-4568},
	edition =   {2nd},
	volume =    {},
}

@book{oksendal2013stochastic,
	title={Stochastic differential equations: an introduction with applications},
	author={Oksendal, Bernt},
	year={2013},
	publisher={Springer Science \& Business Media}
}

@article{saisho1987stochastic,
	title={Stochastic differential equations for multi-dimensional domain with reflecting boundary},
	author={Saisho, Yasumasa},
	journal={Probability Theory and Related Fields},
	volume={74},
	number={3},
	pages={455--477},
	year={1987},
	publisher={Springer}
}

@article{briand2020mean,
	title={Mean reflected stochastic differential equations with jumps},
	author={Briand, Phillippe and Ghannoum, Abir and Labart, C{\'e}line},
	journal={Advances in Applied Probability},
	volume={52},
	number={2},
	pages={523--562},
	year={2020},
	publisher={Cambridge University Press}
}
\bibliographystyle{abbrv}
\end{document}